\documentclass[12pt]{amsart}

\usepackage{amsmath,amsfonts,amssymb,mathrsfs}

\newcommand*{\dgr}[1]{ {\mathrm d}^\circ_{#1} }
\newcommand*{\mi}{{\mathrm i}}
\newcommand*{\e}{{\mathrm e}}
\newcommand*{\dif}{\mathrm{d}}
\newcommand*{\hardy}[1]{{\mathsf H}_{#1}^2}

\newcommand*{\spa}{\hspace*{2em}}
\newcommand*{\proj}[1]{\mathrm{proj}\left({#1}\right)}

\newcommand*{\cnj}[1]{\overline{#1}}

\newtheorem{theorem}{Theorem}
\newtheorem{lemma}{Lemma}
\newtheorem{corollary}{Corollary}
\theoremstyle{definition}
\newtheorem*{remark}{Remark}

\begin{document}

\title{Blaschke products and unwinding in higher dimensions}

\author[R.~Coifman]{Ronald~R.~Coifman} \address{
Department of Mathematics, Program in Applied Mathematics, Yale University, New
Haven, CT 06510, USA} \email{coifman-ronald@yale.edu}
\author[J.~Peyri\`ere]{Jacques~Peyri\`ere} \address{
Institut de Mathématiques d’Orsay, CNRS, Université Paris-Saclay, 91405 Orsay, France}
\email{jacques.peyriere@universite-paris-saclay.fr}
\thanks{The second author gratefully acknowledges the hospitality of
  Yale University where most of this work was done}


\maketitle

\begin{abstract} We give a necessary and sufficient condition for the
  convergence of an infinite product of rational inner functions on
  the polydisk, and explore generalization to the polydisk of
  Malmquist-Takenaka bases and various versions of unwinding.
\end{abstract}

\section{Introduction}

Our goal is to extend to the polydisk various approximation theorems related to Blaschke products and  Malmquist Takanaka unwinding expansions.  We start with a few facts that we are going to generalize 
to several dimensions,  starting with observation of Rudin  that 
any rational inner function in the unit disk (i.,e., analytic and of
absolute value~1 on the boundary) is a finite Blaschke product
\begin{equation}\label{blaschke1d}
  z^k\prod_{1\le j< n}\frac{\cnj{a_j}}{|a_j|}
  \frac{z-a_j}{1-\cnj{a_j}z},
\end{equation}
where $k\ge 0$ and $|a_j|< 1$. If $\sum (1-|a_j|)< \infty$ the
infinite product
$\prod_{j\ge 1}\frac{\cnj{a_j}}{|a_j|} \frac{z-a_j}{1-\cnj{a_j}z}$
converges, uniformly on compact subsets of the unit disk, towards an
inner function, otherwise this product diverges to 0 on
$\mathbb D$~\cite{hoffman}.

More recently \cite{ coifman1,coifman2,mi,qian} obtained several adapted expansions of functions in the Hardy space
$\hardy{}(\mathbb D)$ of the form
\begin{equation}\label{unwinding}
  \sum_{n\ge 0} c_n\prod_{1\le j< n} B_j,
\end{equation}
where the $c_n$ are complex numbers and the $B_j$ are Blaschke products,
have been applied.

\section{Notation}
\hspace*{1em} \bigskip

\begin{tabular}{l|l}
  Expression & meaning\\
  \hline
  $d$ & a positive integer\\
  $\mathbb D$& the unit disk: ${\mathbb D} = \{z\in {\mathbb C}\ :\ |z|<1\}$\\
  $\partial {\mathbb D}$ & the boundary of $\mathbb D$ identified with the torus ${\mathbb T}= {\mathbb R}/2\pi{\mathbb Z}$\\
  $\partial^* {\mathbb D}^d$ & \text{the distinguished
                               boundary of } ${\mathbb D}^d$, identified\\
             &\hspace{2em} with ${\mathbb T}^d = \left({\mathbb R}/2\pi{\mathbb Z}\right)^d$\\
  $\hardy{d}$&  The Hardy space $\hardy{}({\mathbb D}^d)$\\
  $\e^{\mi\theta}$ & $\e^{\mi\theta_1},\e^{\mi\theta_2},\dots,\e^{\mi\theta_d}$\\
  $\dif\theta$ & $\dif\theta_1\dif\theta_2\cdots\dif\theta_d$\\
  $\langle f,g\rangle$ & $\displaystyle (2\pi)^{-d}\int_{\partial^*{\mathbb D}^d} f(\e^{\mi\theta}) \overline{g(\e^{\mi\theta})}\dif\theta$\\
  $\|f\|$ & $\sqrt{\langle f,f\rangle}$\\
  $m\succeq n$ & $m_j\ge n_ j \text{ for } 1\le j\le d$,\\
             &\hspace*{1em}if $m=(m_1,m_2,\dots,m_d)$ and $n=(n_1,n_2,\dots,n_d)$,\\
  $z$ & $z_1,z_2,\dots, z_d$\\
  $z^m$ & $z_1^{m_1}z_2^{m_2}\cdots z_d^{m_d}$, with $m=(m_1,m_2,\dots,m_d)$\\
  $\dgr{i} P$ & the degree of $P\in {\mathbb C}[z]^d$ with respect to the $i$-th variable,\\
  $\dgr{} P$ & $\bigl(\dgr{1} P,\dgr{2} P,\dots,\dgr{d} P \bigr)$,\\
  $\overline{P}$ & the polynomial obtained by replacing in $P$ each\\
             &\hspace*{2em} coefficient by its complex conjugate\\
  $P^*(z)$ & $z^{\dgr{} P}\overline{P}(z_1^{-1},z_2^{-1},\dots,z_d^{-1})$\\
  $\nu(P)$ & $\left\|P^{-1}\right\|^{-1}$\\
\end{tabular}\bigskip

The only polynomials~$P$ we consider, are non constant, have no zeroes
in ${\mathbb D}^d\cup \partial^*{\mathbb D}^d$, and have no common
factor with~$P^*$. Let ${\mathscr P}_d$ stand for this set of
polynomials. In these conditions  $P^*/P$ is inner, which means that it is an analytic function on the
polydisk and has modulus~1 at almost every points of the distinguished
boundary $\partial^*{\mathbb D}^d$. Indeed it is known that any
rational inner function is of the form $az^mP^*P^{-1}$
\cite{rudin1,rudin2,knese}.
\medskip

For $P\in {\mathscr P}_d$ it will be convenient to set
$\alpha(P) = P^*(0)/P(0)$ and
$$\beta(P) = \begin{cases} \cnj{\alpha(P)}/|\alpha(P)& \text{ if } \alpha(P)\ne 0,\\ 1& \text{ otherwise.}\end{cases}
$$

\begin{remark}
  When $d=1$, since the polynomials $P$ and $P^*$ have no common
  factors $P$ has the following form, up to a multiplicative constant,
  $P(z) = \prod_{j=1}^k (1-\overline{a_j}z)$, with $|a_j|< 1$. Then
  $$ \frac{P^*(z)}{P(z)} = \prod_{1\le j\le k} \frac{a_j-z}{1-\overline{a_j}z},
$$
which is a usual finite Blaschke product. So, from now on we call
$P^*/P$ a Blaschke product --- a d-Blaschke product if mentionning the
dimension is needed.
\end{remark}

\section{Blaschke products}

If $P\in {\mathscr P}_d$, the function
$P^*(w,w,\dots,w)/P(w,w,\dots,w)$ is an inner function on ${\mathbb D}$
whose value at~0 is $\alpha(P)$. This implies~$|\alpha(P)|< 1$.

\begin{theorem}\label{convergence} An infinite product
  $\displaystyle \prod_{n\ge 1} \beta(P_n) \frac{P_n^*}{P_n}$
  converges (uniformly on compact subsets of ${\mathbb D}^d$) if and
  only if
  $\displaystyle \sum_{n\ge 1} \Bigl(1-\bigl|\alpha(P_n)\bigr|\Bigr)<
  +\infty.$
\end{theorem}

\proof
Suppose first that
$\displaystyle \sum_{n\ge 1} \Bigl(1-|a(P_n)|\Bigr)< +\infty$. Then
there are finitely many~$n$ such that $\alpha(P_n)=0$. So we may
suppose $\alpha(P_n)\ne 0$ for all~$n$. Set
$\displaystyle B_n = \prod_{1\le j< n} \beta(P_n)\frac{P_n^*}{P_n}$. Take
$m<n$, then
\begin{eqnarray*}
  \int_{{\mathbb T}^d}|B_m-B_n|^2 \dif\theta &=& 2\Re \int_{{\mathbb T}^d} \left( 1-\overline{B_m}B_n\right)\,\dif\theta\\
                                                            &=& 2\Re \int_{{\mathbb T}^d} \left( 1-\frac{B_n}{B_m}\right) \,\dif\theta,
\end{eqnarray*}
(do not forget that $|B_m(\e^{\mi\theta)}| = 1$).
Since
$\displaystyle \frac{B_n}{B_m} = \prod_{m\le j< n}
\beta(P_j)\frac{P_j^*}{P_j}$ is analytic, its mean value on
${\mathbb T}^d$ is its value at~$0$, that is
$\prod_{m\le j< n}|a(P_j)|$. It results that $B_n$ has a limit in
$L^2\bigl(\partial^* {\mathbb D}^d\bigr)$ and converges, uniformly on
any compact subset of the polydisk, to a function which is not~0 since
its value at~$0$ is $\prod_{n\ge 1} \alpha(P_n)$.  \medskip

Now we suppose
$\displaystyle\sum_{j\ge 0} \bigl(1-|\alpha(P_j)|\bigr) = +\infty$.
The $B_n$ form a normal family, We wish to show that~0 is the only
limit point of the sequence $B_n$. Such a limit point~$f$ is analytic
in ${\mathbb D}^d$.

When $d=1$ we know that this infinite product diverges to~0.

Let us first consider the case~$d=2$.

Let $w\in {\mathbb D}$ and $u\in \partial{\mathbb D}$.
Set
$$
g_n(w) = \beta(P_n)\frac{P_n^*(uw,w)} {P_n(uw,w)}.
$$
If $\alpha(P_j)=0$ an infinite number of times the corresponding $g_n$
are Blaschke products vanishing at~0. So the infinite product
$\prod g_n$ diverges. If $\alpha(P_j)=0$ a finite number of times, we
may only consider the $n$ for which $\alpha(P_n)\ne 0$. Then $g_n$ is
a 1-Blaschke product whose value at~0 is $\alpha(P_n)$. So again the
product $\prod g_n$ diverges. This means that for any limit point~$f$
of the sequence~$B_n$ we have $f(uw,w)=0$.

In other terms, $f(z_1,w)=0$ for all~$z_1$ such that $|z_1|=|w|$,
which implies that this again holds if $|z_1|\le |w|$. Therefore
$f=0$.

We are going to prove by induction on~$d$ that such a product
diverges.  Suppose that this is true in $d$ variables. Take
$z= (z_1,\dots,z_{d-1}) \in {\mathbb D}^{d-1}$,
$u\in \partial{\mathbb D}$, and $w\in {\mathbb D}$, and set
$$g_n(z,w) = \beta(P_n)\frac{P_n^*(z,uw,w)}{P_n(z,uw,w)}.$$
Then $g_n$ is a d-Blaschke product whose value at~0 is
$\alpha(P_n)$. So again the product $\prod g_n$ diverges. This means
that for any limit point~$f$ of the sequence~$B_n$ we have
$f(z,uw,w)=0$ for all $z= (z_1,\dots,z_{d-1}) \in {\mathbb D}^{d-1}$,
$u\in \partial{\mathbb D}$, and $w\in {\mathbb D}$. As previously,
this shows that~$f=0$.

\begin{theorem}\label{complet}
  In case $\sum (1-|\alpha(P_n)| = \infty$ we have
  $\displaystyle \bigcap B_n\hardy{d} = \{0\}$.
\end{theorem}

\proof This is obvious when $d=1$ because if $f\in \bigcap B_n\hardy1$
the function $f$ has zeroes which accumulate so quickly to the
boundary that it has to be~$0$. As previously we make a recursion
on~$d$.

Suppose that this is true for~$d$ and consider such a product in
$(d+1)$ variables.  Let $z\in {\mathbb D}^{d-1}$,
$u\in \partial{\mathbb D}$, and $w\in {\mathbb D}$. Set
$B_n^\sharp(z,w)= B_n(z,uw,w)$. The $d$-Blaschke product $B_n^\sharp$
has the same value at~0 as $B_n$. Let $f\in \bigcap
B_n\hardy{d+1}$. Then the function $g(z,w) = f(z,uw,w)$ belongs to
$\bigcap B_n^\sharp\hardy{d}$. It results from the recursion
hypothesis that $g=0$. So we have $f(z,v,w)=0$ when $|v|=|w|$, which
implies that this is still true when $|v|\le |w|$. Finally $f=0$.

\section{Malmquist-Takenaka orthogonal systems}

\subsection{Malmquist-Takenaka orthogonal systems}

\begin{lemma}\label{MK1} Let $P\in {\mathscr P}_d$.  Then
  $f/P$ is orthogonal to $\frac{P^*}{P}\hardy{d}$ if and only if $\widehat{f}(j) = 0$ for all $j$ such that $j\succeq \dgr{} P$.
\end{lemma}

\proof Let $f,\,\varphi\in \hardy{d}$. We have
\begin{equation*}
  \left\langle \frac{P^* \varphi}{P},\frac{f}{P}\right\rangle =
  \left\langle \frac{z^{\dgr{} P} \varphi}{P},f\right\rangle
\end{equation*}
because when $|z|=1$ we have $P^*(z)/\overline{P(z)} = z^{\dgr{} P}$.
Since $\varphi/P$ is any element of~$\hardy{d}$, this scalar product is
null if and only if $f$ is orthogonal to $z^{\dgr P}\hardy{d}$, which
means that the Fourier coefficient $\widehat{f}(j+\dgr{} P) = 0$ for
all $j$ such that $j\succeq 0$ .

\begin{corollary}\label{MK2} Let $(P_n)_{n\ge 1}$ be a sequence of elements
  of ${\mathscr P}_d$. Then
  $\displaystyle \left(\frac{\nu(P_n)}{P_n}\prod_{1\le j<n}
    \frac{P_j^*}{P_J}\right)_{n\ge 1}$ is an orthonormal system in
  $\hardy{d}$.
\end{corollary}

This is reminiscent of the
Malmquist-Takenaka~\cite{malmquist,takenaka} bases in one
dimension. The main difference is that, when $d>1$, this cannot be a
basis, even in case the infinite Blaschke product diverges. Indeed it
results from Lemma~1 that all the spaces
$B_n\hardy{d}\ominus B_{n+1}\hardy{d}$ are not finite dimensional (we
set $B_n = \prod_{1\le j<n} P_j^*/P_j$).

In particular when the degree of~$P$ is $(1,1,\dots,1)$ the space
$\hardy2\ominus B\hardy2$ consists of the functions of the form
$\bigl(\sum_{j=1}^d f_j(z_j)\bigr)/P(z)$.

\section{Orthogonal projections}

In this section we consider orthogonal projections on $B\hardy{d}$,
where $B=P^*/P$.

\subsection{Tensor products of Moebius functions}
\hspace*{1em}\medskip

Lemma~\ref{MK1}, when $d=1$, reduces to Malmquist-Takenaka's theorem:
if $B_a(z_1) = \frac{z_1-a}{1-\cnj{a}z_1}$ the space
  $\hardy1\ominus B_a\hardy1$ is unidimensional and
$\frac{\sqrt{1-|a|^2}}{1-\cnj{a}z_1}$ is an orthonormal basis of
it. Therefore a simple calculation shows that the
orthogonal
    projection of~$f$ on this space is
$$\frac{1-|a|^2}{1-\cnj{a}z_1}\,f(a),$$

\subsubsection{Tensor product of two M\"obius functions}
\hspace*{1em}\medskip

Now
$B(z_1,z_2) =
\frac{z_1-a_1}{1-\cnj{a_1}z_1}\frac{z_2-a_2}{1-\cnj{a_2}z_2}$.  The
projection of $f$ on
$\hardy2 \ominus \frac{z_1-a_1}{1-\cnj{a_1}z_1)}\hardy2$ is (see above)
$$\frac{1-|a_1|^2}{1-\cnj{a_1}z_1}\,f(a_1,z_2).$$
So the projection of~$f$ on $B\hardy2$ is 
$$ \frac{1-|a_2|^2}{1-\cnj{a_2}z_2}\,\left(f(z_1,a_2)-\frac{1-|a_1|^2}{1-\cnj{a_1}z_1}\,f(a_1,a_2)\right)
$$
and the projection on $\hardy2\ominus B\hardy2$ is
$$
\frac{1-|a_2|^2}{1-\cnj{a_2}z_2}\,f(z_1,a_2)+
\frac{(1-|a_1|)^2(1-|a_2|)^2}{(1-\cnj{a_1}z_1)(1-\cnj{a_1}z_1)}\, f(a_1,a_2)
+\frac{1-|a_1|^2}{1-\cnj{a_1}z_1}\,f(a_1,z_2).
$$

\subsubsection{Tensor product of more M\"obius functions}
\hspace*{1em}\medskip

By iterating the previous construction we obtain the following
expression of the orthogonal projection of
$f\in \hardy{d}$ on the orthogonal complement of $\displaystyle
\left(\prod_{j=1}^d \frac{a_j-z_j}{1-\cnj{a_j}z_j}\right)\hardy{d}$:
  
$$\sum_{k=1}^d (-1)^{k-1} \sum^* \left(\prod_{j=1}^k\frac{1-|a_j|^2}{1-\cnj{a_j}z_j}\right)\, f(a_1,\dots,a_{k},z_{k+1},\dots,z_d),
$$
where the symbol~$\sum^*$ means that one has to add analogous terms
obtained by permutations.

\subsection{Kernel for orthogonal projection}

When a direct computation, as above, is not obvious one can use a kernel.
The orthogonal projection on $\hardy{d}$ is given by the Szeg\"o kernel
$$
S_z(z,\zeta) = \prod_{j=1}^d 
\frac{1}{2\pi(1-z_j\overline{\zeta}_j)}.$$

If~$B$ is an inner function on the polydisk the orthogonal projection on
$B\hardy{d}$ is given by the kernel obtained by conjugating the Szeg\"o
operator by multiplication by~$B$:
$$ B(z)S_z(z,\zeta)\overline{B(\zeta)}.
$$
So the kernel for the orthogonal projection o $\hardy{d} \ominus B\hardy{d}$ is
$$
K(z,\zeta) = \frac{1-B(z)\overline{B(\zeta)}}{\prod_{j=1}^d 2\pi(1-z_j\overline{\zeta_j})}.
$$
In other terms
$$
\proj{f,(B\hardy{d}}^\perp)(z) = \int_{{\mathbb T}^d}
K(z,\e^{\mi\theta})\,f(\e^{\mi\theta})\,\dif\theta.
$$

The use of this kernel may lead to tedious computation, even in simple
simple situation as shown by the next example, done with the help of
the Maple software.

\subsubsection{An example}\hspace*{1em}

This time
$P(z_1,z_2) =
1-\overline{a}z_1-\overline{b}z_2+\overline{c}z_1z_2.$. We first consider
the function $g(z_1,z_2)= 1/((1-t_1z_1)(1-t_2z_2))$ because
$\proj{g,(B\hardy2)^\perp}$ is the generating function of
$\proj{z_1^{n_1}z_2^{n_2}, (B\hardy2)^\perp}$.

As the result of a residue calculation we get
{\tiny
\begin{multline*}
  \proj{g,(B\hardy2)^\perp} =\\
  -\frac{\overline{c} b z_1^2+z_1 {| a |}^2-z_1 {| b |}^2-z_1 {| c |}^2-\overline{a} z_1^2+\overline{b} c -a +z_1}{\left(t_1 z_1-1\right) \left(b t_2 z_1-c t_2+a -z_1\right) \left(-z_1 w \overline{c}+z_1 \overline{a}+w \overline{b}-1\right)}\\
  +\frac{\left(a w +b z_1-w z_1-c \right) \left(-\overline{a} c t_2^2+t_2 {| a |}^2-t_2 {| b |}^2+t_2 {| c |}^2+b t_2^2-\overline{c} a +\overline{b}-t_2\right)}{\left(w t_2-1\right) \left(-c t_1 t_2+a t_1+b t_2-1\right) \left(-z_1 w \overline{c}+z_1 \overline{a}+w \overline{b}-1\right) \left(b t_2 z_1-c t_2+a -z_1\right)}
\end{multline*}}
\bigskip

As expected, due to Lemma~\ref{MK1}, we have
\begin{equation*}
  P\times\proj{g,(B\hardy2)^\perp} = \varphi_1(z_1)+\varphi_2(z_2).
\end{equation*}
where

\begin{multline*}
  \left(bt_2 z_1-ct_2+a -z_1\right) \left(t_1 z_1-1\right) \left(-ct_1 t_2+at_1+bt_2-1\right)\varphi_1(z_1) =\\
\left(bt_2-1\right) \left(-\overline{a} ct_1 t_2+| a |^2 t_1-| b |^2 t_1+bt_1 t_2+\overline{c} b -\overline{a}\right) z_1^2+\\
\left(\overline{a} c^2 t_1 t_2^2+\overline{a} bc t_2^2-2 | a |^2 ct_1 t_2+2 | b |^2 ct_1 t_2-bc t_1 t_2^2\right.\\
  -2 | c |^2 bt_2+| a |^2 at_1-| b |^2 at_1-b^2 t_2^2\\+\left.\overline{c} ab -\overline{b} ct_1+| c |^2-| a |^2+at_1+2 bt_2-1\right) z_1\\
-\overline{a} c^2 t_2^2-\overline{b} c^2 t_1 t_2+\overline{b} ac t_1+| c |^2 ct_2+| a |^2 ct_2+ac t_1 t_2+bc t_2^2\\-a | c |^2-a^2 t_1-ab t_2-ct_2+a
\end{multline*}

\begin{equation*}
\varphi_2(z_2) =
-\frac{w \left(-\overline{a} ct_2^2+t_2 | c |^2+t_2 | a |^2-t_2 | b |^2+bt_2^2-\overline{c} a +\overline{b}-t_2\right)}{\left(wt_2-1\right) \left(-ct_1 t_2+at_1+bt_2-1\right)}
\end{equation*}

\section{Unwindings}
\subsection{A first expansion}\hspace*{1em}
\medskip

Given a sequence of polynomials $P_n$ such that $\sum (1-\alpha(P_n)) = \infty$, set $B_n = P_n^*/P_n$. Let $f$ be  a function in $\hardy{d}$. We define by recursion two sequences:\\
$f_0=f$, and,  for $n\ge 0$,\\
$ g_n=\proj{f_n,(B_{n+1}\hardy{d})^\perp}, \text{ and }
f_{n+1}=(f_n-g_n)/B_{n+1}$.

We then have
$$
f = g_o+B_1g_1+B_1B_2g_2+\dots+B_1B_2\cdots  B_n g_n +B_1B_2\cdots B_{n+1}f_{n+1}.
$$
This is a sum of orthogonal terms. Because the product $\prod B_n$
diverges, we get the following orthogomal expansion, converging in
$\hardy{}$~:
$$ f = \sum_{n\ge 0} g_n\prod_{1\le j\le n} B_j.
$$

\subsection{Adaptative unwinding}\hspace*{1em}
\medskip

Consider a compact subset~$\mathscr K$ of~${\mathscr P}$. Given
$f\in \hardy{d}$, let us consider the following procedure,\bigskip

\noindent Set $n=0$, $f_0=f$, and $B_0=1$,\\
\textbf{repeat}\\
\spa\{\\
\noindent\spa
$P_n := \text{argmin} \left\{ \left\|\proj{f_n,P^*P^{-1}\hardy{d}}\right\|\ : P\in {\mathscr K}\right\}$,\\
\spa $U_n = \proj{f_n, (P_n^*P_n^{-1}\hardy{d})^\perp}$\\
\spa $f_{n+1} = f_n-U_n$\\
\spa $B_{n+1} = B_n P_n^*/P_n$\\
\spa increment $n$ by~1.\\
\spa\}.
\bigskip

In this way we get a series $\displaystyle \sum_{n\ge 0}
U_nB_n$ of mutually orthogonal terms. Since all polynomials are chosen
in a compact set, the product
$B_n$ is divergent, which implies $\displaystyle \bigcap B_n\hardy{d} =
\{0\}$. It results that
$$ f = \sum_{n\ge 0} U_nB_n.$$

\subsection{A less greedy  unwinding}\hspace*{1em}

Consider a compact subset~$\mathscr K$ of~${\mathscr P}_d$. Given
$f\in \hardy{d}$, let us consider the following procedure,\medskip

\noindent Set $n=0$ and $B_0 =1$\\
\textbf{repeat}\\
\spa $\displaystyle P_n = \text{argmax}\left\{ \left|\left \langle f,\frac{\nu(P)}{P}B_n\right\rangle\right|\ :  P\in{\mathscr K}\right\}$\\
\spa $\displaystyle a_n = \left\langle f,\frac{\nu(P_n)}{P_n}B_n\right\rangle$\\
\spa $\displaystyle B_{n+1} = B_n\frac{P_n^*}{P_n}$\\
\spa increment $n$ by~1.\bigskip

In this way we get a series $S = \displaystyle \sum_{n\ge 0} a_n
\frac{B_n}{P_n}$ which converges in $\hardy{d}$.
If ${\mathscr K}$ contains only one
polynomial~$P$, there is no chance that this orthogonal series
represents~$f$ ---contrary to what happens in 1D when $\dgr{} P =
1$ --- because the spaces $B_n\hardy{d}\ominus
B_{n+1}\hardy{d}$ are not one dimensional. But, if ${\mathscr
  K}$ is fat enough we could expect that this greedy algorithm
captures most of the energy of~$f$.
\bigskip

\subsection*{Last remark} In the unit disk~${\mathbb D}$ the unwinding
procedure is more accurate.  Indeed, at each step, all the zeroes are
removed. This is possible by the inner/outer factorization theorem and
its efficient implementation by~\cite{Weiss-Weiss}.

\end{document}